\documentclass[11pt]{amsart}
\usepackage{myformat}

\RequirePackage{amssymb}
\RequirePackage{enumerate}
\RequirePackage{amsfonts}
\RequirePackage{amsmath}
\RequirePackage{mathrsfs}
\RequirePackage{dsfont}
\allowdisplaybreaks[1]

\newcounter{fig}

\newcommand{\mybic}{\author{Gianluca Cassese}
                     \address{Universit\`{a} Milano Bicocca}
                     \email{gianluca.cassese@unimib.it}
                     \curraddr{Department of Economics, Statistics and Management 
                                  Building U7, Room 2097, via Bicocca 
                                  degli Arcimboldi 8, 20126 Milano - Italy}}

\newcommand{\A}{\mathscr{A}} 
\newcommand{\F}{\mathscr{F}}
\newcommand{\R}{\mathbb{R}} 
\newcommand{\N}{\mathbb{N}}

\newcommand{\Bor}{\mathscr{B}}

\newcommand{\qtext}[1]{\quad\text{#1}\quad}
\newcommand{\qqtext}[1]{\qquad\text{#1}\qquad}

\newcommand{\abs}[1]{\vert #1\vert}

\newcommand{\Seq}[2][]{\mathfrak S_{#1}(#2)}

\newcommand{\SeqUp}[2][]
	{\mathfrak S^{\uparrow}_{#1}(#2)}
\newcommand{\SeqDn}[2][]
	{\mathfrak S^{\downarrow}_{#1}(#2)}
\newcommand{\net}[3]{\langle #1_{#2}\rangle_{#2\in #3} }

\newcommand{\seq}[2]{\net{#1}{#2}{\mathbb{N}}} 
 
\newcommand{\seqn}[1]{\seq{#1}{n}}

\newcommand{\norm}[1]{\Vert #1\Vert}

\newcommand{\sset}[1]{\mathds{1}_{\{#1\}}}

\newcommand{\emp}{\varnothing}

\newcommand{\tiref}[1]{(\textit{#1})}

\makeatletter
\let\save@mathaccent\mathaccent
\newcommand*\if@single[3]{%
  \setbox0\hbox{${\mathaccent"0362{#1}}^H$}%
  \setbox2\hbox{${\mathaccent"0362{\kern0pt#1}}^H$}%
  \ifdim\ht0=\ht2 #3\else #2\fi
  }
\newcommand*\rel@kern[1]{\kern#1\dimexpr\macc@kerna}
\newcommand*\widebar[1]{\@ifnextchar^{{\wide@bar{#1}{0}}}{\wide@bar{#1}{1}}}
\newcommand*\wide@bar[2]{\if@single{#1}{\wide@bar@{#1}{#2}{1}}{\wide@bar@{#1}{#2}{2}}}
\newcommand*\wide@bar@[3]{%
  \begingroup
  \def\mathaccent##1##2{%
    \let\mathaccent\save@mathaccent
    \if#32 \let\macc@nucleus\first@char \fi
    \setbox\z@\hbox{$\macc@style{\macc@nucleus}_{}$}%
    \setbox\tw@\hbox{$\macc@style{\macc@nucleus}{}_{}$}%
    \dimen@\wd\tw@
    \advance\dimen@-\wd\z@
    \divide\dimen@ 3
    \@tempdima\wd\tw@
    \advance\@tempdima-\scriptspace
    \divide\@tempdima 10
    \advance\dimen@-\@tempdima
    \ifdim\dimen@>\z@ \dimen@0pt\fi
    \rel@kern{0.6}\kern-\dimen@
    \if#31
      \overline{\rel@kern{-0.6}\kern\dimen@\macc@nucleus\rel@kern{0.4}\kern\dimen@}%
      \advance\dimen@0.4\dimexpr\macc@kerna
      \let\final@kern#2%
      \ifdim\dimen@<\z@ \let\final@kern1\fi
      \if\final@kern1 \kern-\dimen@\fi
    \else
      \overline{\rel@kern{-0.6}\kern\dimen@#1}%
    \fi
  }%
  \macc@depth\@ne
  \let\math@bgroup\@empty \let\math@egroup\macc@set@skewchar
  \mathsurround\z@ \frozen@everymath{\mathgroup\macc@group\relax}%
  \macc@set@skewchar\relax
  \let\mathaccentV\macc@nested@a
  \if#31
    \macc@nested@a\relax111{#1}%
  \else
    \def\gobble@till@marker##1\endmarker{}%
    \futurelet\first@char\gobble@till@marker#1\endmarker
    \ifcat\noexpand\first@char A\else
      \def\first@char{}%
    \fi
    \macc@nested@a\relax111{\first@char}%
  \fi
  \endgroup
}
\makeatother

\newcommand {\FacSeq}[1]{\widebar{\Seq{#1}}}

\newcommand {\Meas}{\mathscr M}    
    
\newcommand {\PN}{\mathcal P(\N)}   
\newcommand {\Aa}{\mathfrak A}   
\newcommand {\G}{\mathscr G}   
\newcommand {\D}{$(\mathbf D)$}   
\newcommand {\Du}{$(\mathbf D_*)$}   
\newcommand {\CC}{$(CC)$}   
\newcommand {\DO}{$(\mathbf D_0)$}   
  
\begin{document} 
\title[Control measures]
{Control measures on Boolean algebras}
\mybic
\date \today 


\begin{abstract} 
In this paper we discuss the existence of a control
measure for a family of measures on a Boolean
algebra. We obtain a necessary and sufficient
condition and several related results, including
a new criterion for weak compactness for
additive set functions on an algebra of sets.
\end{abstract}

\maketitle

\begin{center}
{\it\Small This paper is dedicated to the memory 
of Fiamma Galgani.}
\end{center}

\section{Introduction.}
In 1947, Dorothy Maharam \cite{maharam} 
introduced and characterised the notion of 
measure algebra, namely a Boolean 
algebra $\A$ endowed with a measure that is 
strictly positive on $\A\setminus\{0\}$. 
Obtaining a characterization of measure 
algebras has since then become a major 
topic of research in measure theory.

In this paper we investigate the somehow related 
question of finding necessary and sufficient 
conditions for a set $\Meas$ of measures on 
$\A$ to admit a dominating or control measure, 
i.e. a measure $\nu$ such that
\begin{equation}
\lim_n\nu(a_n)=0
\qtext{implies}
\lim_n\mu(a_n)=0
\qquad
\mu\in\Meas.
\end{equation}

Although in the measure algebra literature 
domination has hitherto played a minor role, 
it has attracted much attention in analysis,
particularly in the study of vector measures 
in which, following Bartle, Dunford and Schartz 
\cite{bartle_dunford_schwartz}, if $F:\A\to X$ 
is an additive function with values in a vector 
space $X$, the existence of a control measure 
for the set $\Meas=\{x^*F:x^*\in X^*\}$ is 
particularly useful.

The problem addressed in this paper has a 
fairly natural translation in the language of 
vector lattices where the domination property 
is reformulated into the condition that a given 
set  belongs to some principal projection band. 
This general problem is fully settled in section 
\ref{sec Banach}. Nevertheless when it comes
to additive functions on a Boolean algebra the 
characterization so obtained is not very explicit 
about the role of the underlying algebra. In
order to obtain a more informative condition
involving $\A$ we introduce a hierarchy of different 
properties concerning $\Meas$, the \DO, \D\ 
and \Du\ properties.
One may consider each of these definitions
as a variant of the well known and historically 
important \CC\ condition. In sections \ref{sec DO} 
and \ref{sec D} we use the first two properties
to study monotone and additive functions 
respectively, while in section \ref{sec D*} we 
characterize weak compactness via the \Du\ 
property. Eventually, in section \ref{sec close} 
we construct a fairly general Boolean algebra
for which the \CC\ condition is necessary and 
sufficient to be a measure algebra. 

All of our results rely on two general lemmas 
of their own interest proved in section 
\ref{sec Boole} for general Boolean algebras.

\subsection{Notation}

In the sequel $\A$ is a Boolean algebra and, 
following \cite{sikorski}, we denote binary 
operations on $\A$ with set theoretic symbols.
Thus, $a\cap b$, $a\cup b$ and $a^c$ denote
meet, join and complementation; we also write 
$a\cap b^c$ as $a\setminus b$ and $1$ and $0$
for the greatest and the least elements.  By a measure 
on $\A$ we mean a function $\mu:\A\to\R_+$ 
such that
\begin{equation}
\mu(a\cap b)+\mu(a\cup b)
	=
\mu(a)+\mu(b)
\qquad
a,b\in\A.
\end{equation}
Of course $\A$ may well be regarded as an 
algebra of subsets of some given set $\Omega$, 
via Stone isomorphism. This remark makes
available several results originally established
for functions defined on an algebra of sets, at
least as long as one avoids infinite operations 
which are generally not preserved under Boolean 
isomorphisms, as is well known (see example 3.1 
in \cite{hewitt}). 

The importance of a Boolean algebra structure
emerges as we proceed to embed in an obvious 
way $\A$ into the Boolean algebra $\Seq\A$ of 
sequences from $\A$. The $n$-th coordinate 
projection of $\sigma\in\Seq\A$ will be denoted 
by $\sigma(n)$ and the range of $\sigma$ by 
$[\sigma]$. The Boolean operations on $\Seq\A$ 
will be denoted by $\sigma\wedge\tau$,
$\sigma\vee\tau$ and $\sigma\sim\tau$ defined
implicitly by letting
\begin{equation}
\big(\sigma\wedge\tau\big)(n)
=
\sigma(n)\cap\tau(n),
\ 
\big(\sigma\vee\tau\big)(n)
=
\sigma(n)\cup\tau(n)
\qtext{and}
\big(\sigma\sim\tau\big)(n)
=
\sigma(n)\setminus\tau(n)
\qquad n\in\N.
\end{equation}
$\SeqDn\A$ (resp. $\SeqUp\A$) will indicate the 
family of decreasing (resp. increasing) sequences 
on $\A$.

Denoting by $\Seq[0]\A$ the ideal of sequences
with finitely many non null elements, we obtain
the factorial Boolean algebra
$\FacSeq\A=\Seq\A/\Seq[0]\A$. 
The image of $\sigma\in\Seq\A$ under the 
canonical isomorphism of $\Seq\A$ into 
$\FacSeq\A$ will be denoted by $\widebar\sigma$. 
Boolean operations on $\FacSeq\A$ and on 
$\Seq\A$ will be indicated by the same symbols.
Every function $m:\A\to\R$ corresponds to a
function $\widebar m:\FacSeq\A\to\R$ via the
equation
\begin{equation}
\label{bar m}
\widebar m(\widebar\sigma)
=
\limsup_nm(\sigma(n))
\qquad
\widebar\sigma\in\FacSeq\A,\ 
\sigma\in\widebar\sigma.
\end{equation}

\section{Banach lattice preliminaries}
\label{sec Banach}

In this section we study the notion of domination
in the context of a given Banach lattice $X$ with 
order continuous norm. Terminology and notation 
are borrowed from \cite{aliprantis_burkinshaw}:
$I_A$ is the ideal generated by $A\subset X$,
$B_x$ is the band generated by $x\in X$
and $\widebar A$ denotes the norm closure 
of $A$.

\begin{lemma}
\label{lemma AL}
If $I\subset X$ is an ideal which contains no 
uncountable collection of non null, pairwise 
orthogonal elements then $I\subset B_z$ 
for some $z\in\widebar{I}$. If $X$ is an abstract 
$L^p$ space ($1\le p<\infty$), the converse 
is also true.
\end{lemma}

\begin{proof}
The family 
$\{A\subset I:
x\perp y\text{ for every }x,y\in A\}$
admits, by Zorn lemma,  a maximal element 
(relative to inclusion) which, by assumption, 
may be enumerated as $x_1,x_2,\ldots$. Then
\begin{equation}
\label{z}
z
=
\sum_n2^{-n}\frac{\abs{x_n}}{1+\norm{x_n}}
\in
\widebar{I}.
\end{equation}
If $x\in I\setminus B_z$ then there
exists $0<y\le\abs x$ orthogonal to 
$x_n$ for $n=1,2,\ldots$, contradicting the
maximality of $\{x_n:n\in\N\}$. Conversely, 
let $I\subset B_z$ for some 
$z\in X_+$ and let
$\{y_\alpha:\alpha\in\Aa\}
	\subset
I_+$ 
be pairwise orthogonal. If $X$ is an 
abstract $L^p$ space, 
$\norm z^p
	\ge
\sum_\alpha\norm{z\wedge y_\alpha}^p$ 
so that $y_\alpha\perp z$ -- and thus 
$\norm{y_\alpha}=0$ -- for all save 
countably many $\alpha\in\Aa$. 
\end{proof}

Thus for a set $A$ in a Banach lattice with order 
continuous norm, $A\subset B_x$ for some $x\in X$ 
if and only if $A\subset B_z$ for some $z$ of the 
form $z=\sum_na_n\abs{x_n}$, with 
$x_1,x_2,\ldots\in A$. In the setting of countably 
additive set functions on a $\sigma$ algebra of
sets this claim was proved by Halmos and 
Savage \cite[Lemma 7]{halmos_savage} (but 
see also Walsh \cite[Lemma 1]{walsh}) while 
its proof in the finitely additive case was 
given in \cite[Theorem 2]{JMAA_2016}. A 
similar property has also been studied recently 
by Lipecki \cite{lipecki_2018} under the name
of band domination.

Although Lemma \ref{lemma AL} provides a clear
answer to the question of a dominating element in
several interesting situations, in the case of a family
of additive set functions it is not particularly informative 
concerning the role underlying family of sets. This 
notwithstanding, Lemma \ref{lemma AL} provides 
a first result on measure algebras, at least in a rather
special case.

\begin{corollary}
Let $A\subset X$ and let $\A$ be the algebra 
of subsets of $A$ generated by the order intervals 
$(0,\abs a]$ with $a\in A$. If $\A$ is a measure algebra 
then $A\subset B_z$ for some $z\in\widebar{I_A}$.
\end{corollary}

\begin{proof}
Let the finitely additive probability $\mu$ on $\A$
be strictly positive on $\A\setminus\{\emp\}$ 
and $\{x_\alpha:\alpha\in\Aa\}$ a disjoint family
in $A\setminus\{0\}$. Then the intervals 
$(0,\abs{x_\alpha}]$ are pairwise disjoint so that
$
\mu(A)\ge\sum_\alpha\mu\big((0,\abs{x_\alpha}]\big)
$
and $\Aa$ is countable. The claim
follows from Lemma \ref{lemma AL}.
\end{proof}

\section{Boolean algebra preliminaries}
\label{sec Boole}

In this section we shall prove two useful lemmas 
on Boolean algebras%
\footnote{
The results of this section may be proved in more
general structures than Boolean algebras.
}
.

\begin{lemma}
\label{lemma Boole1}
Let a Boolean algebra $\A$ with each countable 
subset admitting an upper bound. Let 
$\emp\ne\G\subset\F\subset\A$ 
be such that
(a) 
$0\notin\F$,
(b) 
$x\in\F$ and $y\ge x$ imply $y\in\F$ and
(c)
any family
$\{x_\alpha:\alpha\in\Aa\}\subset\G$
with $x_\alpha\wedge x_{\alpha'}\notin\F$
when $\alpha\ne\alpha'$ is at most countable.
Then for some $x_0\in\F$
\begin{equation}
x\sim x_0\notin\G
\qquad
x\in\G.
\end{equation}

\end{lemma}

\begin{proof}
Let $\Gamma$ be a choice function associating 
each $\sigma\in\Seq\A$ with an upper bound of
$[\sigma]$. If 
$\sigma(1)\in\F$ then,
$\Gamma(\sigma)
	\ge
\sigma(1)
$
implies $\Gamma(\sigma)\in\F$; if 
$x\in[\sigma]$ then 
$x\sim\Gamma(\sigma)\notin\G$.
Define
\begin{equation}
\Aa
=
\big\{(x,\sigma)
\in
\G\times\Seq{\G}:
x\sim\Gamma(\sigma)\in\G
\big\}.
\end{equation}
If $\Aa$ is empty, the claim is trivial. Otherwise, 
write $(y,\tau)\succ (x,\sigma)$ to indicate that 
$\{x\}\cup[\sigma]\subset[\tau]$. Let $\Aa_0\subset\Aa$
be a maximal, linearly $\succ$ ordered subset. 
If $(x,\sigma),(y,\tau)\in\Aa_0$ and, say, 
$(y,\tau)\succ(x,\sigma)$ 
then 
$
(y\sim\Gamma(\tau))
\wedge
(x\sim\Gamma(\sigma))
	\le
x\sim\Gamma(\tau)
	\notin
\F
$.
Thus, the collection
$\{x\sim\Gamma(\sigma):
(x,\sigma)\in\Aa_0\}
	\subset
\G$
is such that the meet of any two elements 
in it does not belong to $\F$ and, by property 
\tiref c, $\Aa_0$ must be countable. Choose
$\sigma_0\in\Seq{\G}$ such that
$[\sigma_0]
=
\bigcup_{(x,\sigma)\in\Aa_0}[\sigma]$
and set 
$x_0=\Gamma(\sigma_0)\in\F$. 
Then, $x\sim x_0\notin\G$ for all $x$
such that $(x,\sigma)\in\Aa_0$. If
$y_0\sim x_0\in\G$
for some $y_0\in\G$, this would imply
$(y_0,\sigma_0)\notin\Aa_0$ and
$(y_0,\sigma_0)\succ(x_0,\sigma)$
for all $(x,\sigma)\in\Aa_0$, 
contradicting the maximality of $\Aa_0$.
\end{proof}

If e.g. $\G=\F=\{x\in\A:\phi(x)>0\}\ne\emp$ with
$\phi:\A\to\R$ an increasing function with
$\phi(0)\le 0$ and $\phi(a)\le\phi(b)+\phi(a\sim b)$, 
then, under the conditions of the Lemma, $\phi$ 
admits a maximum.

We shall make use of property \tiref{c} of Lemma 
\ref{lemma Boole1} sufficiently often to justify
referring to that condition by saying that $\G$
is sparse in $\F$. Properties \tiref{a} and \tiref{b}
imply that writing
\begin{equation}
\label{>F}
y>_\F y
\qtext{whenever}
y\ge x
\qtext{and}
y\sim x\in\F
\end{equation}
implicitly defines an asymmetric partial order.

\begin{remark}
Although in general a Boolean algebra $\A$ may
fail to satisfy the condition on the existence of
upper bounds for countable subsets stated in 
Lemma \ref{lemma Boole1}, this property holds
in $\FacSeq\A$. In fact, if 
$\{\widebar\sigma_n:n\in\N\}
\subset\FacSeq\A$
and if $\sigma_n\in\widebar\sigma_n$
for each $n\in\N$, let $\sigma,\tau\in\Seq\A$ 
be defined via
\begin{equation}
\upsilon(n)
=
\bigcup_{j\le n}\sigma_j(n)
\qtext{and}
\tau(n)
=
\bigcap_{j\le n}\sigma_j(n)
\qquad
n\in\N.
\end{equation}
Then $\widebar\upsilon$ is an upper bound and
$\widebar\tau$ a lower bound for 
$\{\widebar\sigma_n:n\in\N\}$.
\end{remark}

\begin{lemma}
\label{lemma Boole2}
Let $\F\subset\A$ be sparse in itself and satisfy
properties (a) and (b) of Lemma \ref{lemma Boole1}.
Any $\G\subset\A$ linearly $>_\F$ ordered either 
admits an $>_\F$ maximum (resp. minimum) or a 
countable subset $\G_0$ having the same upper 
(res. lower) $\ge$ bounds as $\G$.
\end{lemma}

\begin{proof}
Put $\Aa=\{(x,x')\in\G\times\G:x'>_\F x\}$. For 
$(x,x'),(y,y')\in\Aa$ write $(y,y')\succ(x,x')$ when 
$y\ge x'$ and let $\Aa_0\subset\Aa$ be a maximal,
linearly $\succ$ ordered subset. If 
$(x,x'),(y,y')\in\Aa_0$ then
\begin{equation}
(x'\sim x)\cap(y'\sim y)
\le
x'\sim y
\le
y\sim y
\notin\F.
\end{equation}
Given that $\F$ is sparse, $\Aa_0$ must be countable
as well as $\G_0=\{x,x':(x,x')\in\Aa_0\}$. If $z_0\in\G$
is an upper bound for $\G_0$ but not for $\G$, then
there exists $z\in\G$ such that $z>_\F x$ for all
$x\in\G_0$. If $z$ is not an $>_\F$ maximum for 
$\G$, then there exists $z'\in\G$ such that
$z'>_\G z$ and therefore such that $(z,z')\in\Aa$
and $(z,z')\succ(x,x')$ for all $(x,x')\in\Aa_0$,
a contradiction.
\end{proof}

\section{Monotonic set functions}
\label{sec DO}

In this section we fix $\Bor\subset\A$ closed 
under $\cup$, $\cap$ and $\setminus$. If
$\psi:\Bor\to\R_+$ we define 
$\widebar\psi:
\FacSeq\Bor\to\R_+\cup\{+\infty\}$ 
as in \eqref{bar m}.

\begin{definition}
A function $\psi:\Bor\to\R_+$ possesses 
property \DO\ if any  collection 
$\{\sigma_i:i\in I\}
\subset
\Seq\Bor$
satisfying 
\begin{equation}
\label{DO}
\inf_{i\in I}\widebar\psi\big(\widebar\sigma_i\big)
	>
0
\qtext{and}
\widebar\psi\big(\widebar\sigma_i\cap\widebar\sigma_j\big)
	=
0
\qquad
i,j\in I,i\ne j,
\end{equation}
is at most countable. 
\end{definition}

\begin{theorem}
\label{th dense psi}
Let $\psi:\Bor\to\R_+$ be monotonic with 
$\psi(0)=0$. If $\psi$ satisfies 
property \DO\ then there exists 
$\sigma_0\in\SeqUp\Bor$ such that
\begin{equation}
\label{dense psi}
\lim_n\psi\big(\sigma(n)\setminus\sigma_0(n)\big)
	=
0
\qquad
\sigma\in\Seq\Bor.
\end{equation}
\end{theorem}

\begin{proof}
Let 
$
\F
	=
\big\{\widebar\sigma\in\FacSeq\Bor:
\widebar\psi(\widebar\sigma)>0\big\}
$.
If a family $\{\widebar\sigma_i:i\in I\}\subset\F$ is such 
that $\widebar\sigma_i\wedge\widebar\sigma_j\notin\F$ 
and $\widebar\psi(\widebar\sigma_i)>0$ with
$I$ uncountable, there must then be 
$\eta>0$ and $I_0\subset I$ uncountable 
such that 
$\inf_{i\in I_0}\widebar\psi(\widebar\sigma_i)
	>
\eta$,
contradicting \eqref{DO}. Thus $\F$ is sparse in 
itself  and, in view of the preceding remark, the 
conditions of Lemma \ref{lemma Boole1} are 
satisfied with $\F=\G$. We deduce the existence 
of $\widebar\sigma_0\in\F$ such that
\begin{equation}
\widebar\psi(\widebar\sigma\sim\widebar\sigma_0)
	=	
0
\qquad
\widebar\sigma\in\F
\end{equation}
which holds trivially even when 
$\widebar\sigma\notin\F$. Of course, since $\psi$ 
is monotonic, the above conclusion still 
holds if we replace each set $\sigma_0(n)$ with 
$\bigcup_{j\le n}\sigma_0(j)\in\Bor$, 
so as to make the sequence increasing. 
\end{proof}

Loosely speaking, one may interpret Theorem 
\ref{th dense psi} as asserting that the sequence 
$\sigma_0$ summarizes most of the relevant 
information conveyed by $\psi$. 
Notice that if $\Bor$ is a $\sigma$ ring then
\eqref{dense psi} implies 
\begin{equation}
\label{sigma}
\lim_k
\psi\Big(\bigcup_{n>k}\sigma_0(n)\setminus\sigma_0(k)\Big)
	=
0.
\end{equation}

We provide examples in which condition 
\eqref{DO} may fail or take a rather
special form.

\begin{example}
Let the range of $\psi$ in Theorem 
\ref{th dense psi} be a finite set (e.g. when
$\psi$ is the supremum of a set of $0-1$ 
valued additive functions). Then for
each pair $i\ne j$,
$\widebar\psi(\widebar\sigma_i\wedge\widebar\sigma_j)
=
0$
if and only if
$\psi(\sigma_i(n)\cap\sigma_j(n))
=
0$
for $n$ sufficiently large. 
Fix $\eta>0$ and let $\{\widebar\sigma_i:i\in I\}$ 
be a maximal (with respect to inclusion) 
set in $\FacSeq\Bor$ satisfying 
\begin{align}
\label{diag}
\inf_{i\in I}\widebar\psi(\widebar\sigma_i)
	>
\eta
\qtext{and}
\widebar\psi(\widebar\sigma_i\cap\widebar\sigma_j)
	=
0
\qquad
i,j\in I,i\ne j.
\end{align}
Under the assumptions of Theorem 
\ref{th dense psi}, $I$ is at most countable. 
However, if $I$ is countably infinite, we may choose
iteratively $n_k$ such that 
\begin{equation}
n_k>n_{k-1}\vee k,\ 
\inf_{n>n_k}\psi(\sigma_k(n))
>
\eta
\qtext{and}
\sup_{n\ge n_k}\sup_{i<k}
\psi(\sigma_i(n)\cap\sigma_k(n))
=
0
\end{equation}
and define 
$\sigma_0(k)
	=
\sum_k\sigma_k(n)\sset{n_k\le n<n_{k+1}}$. 
Then \eqref{diag} extends to 
$\{\widebar\sigma_i:i\in I\}\cup\{\widebar\sigma_0\}$, 
contradicting the maximality of $I$. 
In other words, in the special case under 
consideration a collection $\{\widebar\sigma_i:i\in I\}$ 
as in Theorem \ref{th dense psi} is at most 
countable if and
only if it is finite.
\end{example}

The following example is related to weak 
compactness, as will be clear after Theorem
\ref{th D*}. Two sequences 
$\sigma,\tau\in\Seq\Bor$ 
are said to be quasi disjoint if 
$\sigma\wedge\tau\in\Seq[0]\Bor$,
i.e. if $\widebar\sigma\wedge\widebar\tau=0$.

\begin{example}
\label{ex N}
Let $\Bor=\PN$ in Theorem \ref{th dense psi}.
By a diagonal argument the maximal family
$\{B^i:i\in I\}$ of infinite subsets of $\N$ with 
finite pairwise intersection is uncountable. If
we write $\sigma_i(n)=B^i\cap\{n,n+1,\ldots\}$
and denote by $\sigma_i$ the corresponding 
sequence, we obtain an uncountable, pairwise 
quasi disjoint family
$\{\sigma_i:i\in I\}\subset\SeqDn\N$. 
By quasi disjointness, 
$\widebar\psi(\widebar\sigma_i\wedge\widebar\sigma_j)=0$
for all $i,j\in I$ with $i\ne j$. Thus in order
for $\psi$ to be of class \DO, we need to have
$\widebar\psi(\sigma_i)=0$ for all save countably
many $i\in I$. In fact this conclusion holds
under a weaker condition than property \DO\ 
that will be introduced in the next section
as property \D.
\end{example}

\section{Additive set functions.}
\label{sec D}

In this section we fix a given family $\Meas$
of measures on $\A$. Our purpose is to obtain a 
characterization of dominated sets of measures 
that may be given entirely in terms of the underlying 
algebra $\A$. The following property is the one 
considered in Example \ref{ex N}.

\begin{definition}
\label{def D}
$\mathscr M$ possesses property \D\ if every 
pairwise quasi disjoint collection 
$\{\sigma_\alpha:\alpha\in\Aa\}
	\subset
\SeqDn\A$ 
satisfying
\begin{equation}
\label{D}
\sup_{\mu\in\mathscr M}
\widebar\mu(\widebar\sigma_\alpha)
	>
0
\qquad
\alpha\in\Aa
\end{equation}
is at most countable. If the same conclusion
holds with \eqref{D} replaced by the weaker 
condition
\begin{equation}
\label{D*}
\lim_n\sup_{\mu\in\Meas}\mu(\sigma_\alpha(n))
	>
0
\qquad
\alpha\in\Aa
\end{equation}
then $\Meas$ is said to be of class \Du.
\end{definition}

In case the elements of $\mathscr M$ are 
countably additive and $\A$ a $\sigma$ 
algebra of subsets of some set $\Omega$, 
each sequence $\sigma_\alpha$ in Definition 
\ref{def D} may be replaced with the element 
$b_\alpha
	=
\bigcap_n\sigma_\alpha(n)$. 
Property \D~ takes then a somewhat easier 
form: each pairwise disjoint collection 
$\{b_\alpha:\alpha\in\Aa\}$ in $\A$ with 
$\sup_{\mu\in\mathscr M}\mu(b_\alpha)
	>
0$
is at most countable. This weaker version of 
property \D\  was introduced  long ago in the 
literature under the name of ``countable chain'' 
\CC\ condition by Maharam \cite[p. 160]{maharam} 
in her study of measure algebras and plays an 
important role in the papers by Musia\l\ 
\cite{musial} and Drewnowski \cite{drewnowski_74} 
(who credits Dubrovski\u{i} \cite{dubrovskii} 
for its first formulation)%
\footnote{
Maharam, differing from the other authors cited,
 considers this condition in the case in which 
 $\Meas$ is the set of {\it all} measures on $\A$.
Drewnowski, \cite[Theorem 2.3]{drewnowski_74}
and Musia\l\ \cite[Theorem 2]{musial}, prove that 
\CC\ is necessary and sufficient for a countably 
additive measure with values in a locally convex 
vector space to admit a control measure. Their 
claim may be easily adapted to show that such 
condition is necessary and sufficient for a dominated 
set of countably additive set functions on a 
$\sigma$ algebra to be dominated, a result 
rediscovered in \cite[Theorem 3]{JMAA_2016}
and whose proof is an immediate corollary of 
the following Theorem \ref{th D} of the present 
paper. I am grateful to professor Lipecki who, 
in a private communication, called my attention 
on these references giving me the opportunity to 
acknowledge the results obtained by a group of 
outstanding mathematicians whose work is 
perhaps too little known.
}.

It should be mentioned that the need for an
extension from {\it sets} to {\it families of sets}, 
exemplified in the shift from property \CC\ to 
property \D, was already clear to Maharam 
who formulated ``postulate II'' (p. 159) as a 
reinforcement of property \CC. In another 
paper on measure algebras, Kelley \cite{kelley} 
considered families of sets with positive 
intersection number. More comments on 
the relationship with the measure algebra 
literature will appear in the closing section 
of the paper.

Before moving to the general implications of
these definitions, three elementary facts may 
be easily established. 

(1). 
A set $\Meas$ consisting of a single element 
$\nu$ possesses property \D. In fact for each 
sequence $\sigma_\alpha$ as in Definition 
\ref{def D} one may let
\begin{equation}
\nu_\alpha(b)
	=
\lim_n\nu(a\cap\sigma_\alpha(n))
\qquad
b\in\A
\end{equation}
obtaining a family $\{\nu_\alpha:\alpha\in\Aa\}$ 
of pairwise orthogonal, non null elements 
contained in the ideal generated by $\nu$, 
so that $\Aa$ must be countable, by Lemma 
\ref{lemma AL}.  

(2).
Thus every dominated set possesses property \D. 
This  same conclusion is no longer valid if $\nu$
dominates $\Meas$ weakly (i.e. $\nu(a)=0$
implies $\mu(a)=0$ for all $\mu\in\Meas$) 
as this  latter condition is not sufficient to infer 
from \eqref{D} that 
$\lim_n\nu(\sigma_\alpha(n))>0$.

In this section we have a special interest for 
those subfamilies of $\FacSeq\A$ in restriction 
to which $\widebar\mu$ is additive. An important 
such class is 
\begin{equation}
\label{S0(A)}
\Sigma^{\downarrow}(\A)
=
\SeqDn\A/\Seq[0]\A.
\end{equation}
Another one, given a measure $\nu$ on $\A$, is the 
subclass 
$\Sigma^{\downarrow}(\A,\nu)
	\subset
\Sigma^{\downarrow}(\A)$ 
obtained upon replacing $\SeqDn\A$ in \eqref{S0(A)} 
with 
$\SeqDn{\A,\nu}
=
\SeqDn{\A}\cap\Seq{\A,\nu}$
where
\begin{equation}
\Seq{\A,\nu}
	=
\Big\{\sigma\in\Seq\A:\lim_n2^n\sup_{k>n}
\nu\big(\sigma(n)\bigtriangleup\sigma(k)\big)
	=
0
\Big\}
\end{equation}
is the Boolean algebra of sequences with exponential 
rate of $\nu$-convergence.

Both $\Sigma^{\downarrow}(\A)$ and 
$\Sigma^{\downarrow}(\A,\nu)$ contain 
the zero and the unit of $\FacSeq\A$ and 
are closed with respect to join and meet. 
Since $\widebar\mu$ is additive on 
$\Sigma^{\downarrow}(\A)$ it is then so 
also on the algebra $\Sigma(\A)$ generated 
by $\Sigma^{\downarrow}(\A)$ 
\cite[p. 478]{horn_tarski} and, {\it a fortiori}, 
on $\Sigma(\A,\nu)$, the algebra generated 
by $\Sigma^{\downarrow}(\A,\nu)$. Moreover, 
since $\Seq{\A,\nu}$ is a Boolean algebra, then 
$\Sigma(\A,\nu)
	\subset
\Seq{\A,\nu}/\Seq[0]\A
$.

(3).
A final simple conclusion is obtained in the following:

\begin{lemma}
\label{lemma Sn}
Let $\nu$ be a measure on $\A$, $\varepsilon>0$ and 
$\widebar\sigma\in\Sigma(\A,\nu)$. There are 
$\widebar\tau,\widebar\upsilon^c
\in
\Sigma^{\downarrow}(\A)$ 
such that
\begin{equation}
\widebar\tau\le\widebar\sigma\le\widebar\upsilon
\qtext{and}
\widebar\nu(\widebar\tau)+\varepsilon
	\ge
\widebar\nu(\widebar\sigma)
	\ge
\widebar\nu(\widebar\upsilon)-\varepsilon.
\end{equation}
\end{lemma}

\begin{proof}
Pick $\sigma\in\widebar\sigma$ and fix $N$ large enough so that 
\begin{equation}
2^{-N}<\varepsilon/2
\qtext{and}
\sup_{k\ge n\ge N}
\nu\big(\sigma(n)\bigtriangleup\sigma(k)\big)
	<
2^{-n}.
\end{equation}
Define $\tau(n)=1$ and
$\upsilon(n)=0$ if $n<N$ or else 
$\tau(n)=\bigcap_{N\le j\le n}\sigma(j)$
and
$\upsilon(n)=\bigcup_{N\le j\le n}\sigma(j)$.
Clearly, $\widebar\tau\le\widebar\sigma\le\widebar\upsilon$
and $\tau,\upsilon^c\in\SeqDn\A$. If $k\ge N$
then
\begin{align*}
\tau(k)\bigtriangleup\sigma(k)
	\subset
\bigcup_{j=N}^{k-1}\sigma(j)\bigtriangleup\sigma(k),
\quad 
\upsilon(k)\bigtriangleup\sigma(k)
	\le
\bigcup_{j=N}^{k-1}\sigma(j)\bigtriangleup\sigma(k)
\end{align*}
and
$
\nu\big(\bigcup_{j=N}^{k-1}\sigma(j)\bigtriangleup\sigma(k)\big)
\le
2^{-(N-1)}
<
\varepsilon
$.
\end{proof}

We shall use the notation $\nu\perp\mu$ and $\nu\ll\mu$
in exactly the same sense as for set functions.

\begin{proposition}
\label{propo D}
Let $\Meas$ possess property \D. Choose a measure
$\nu$ on $\A$ such that $\nu\perp\mu$ for every 
$\mu\in\Meas$ and fix $0<t<1$. Then there exists 
$\tau_*\in\SeqDn\A$ such that
\begin{equation}
\lim_n\nu(\tau_*(n))
	\ge 
(1-t)\norm\nu
\qtext{while}
\sup_{\mu\in\Meas}\lim_n\mu(\tau_*(n))
	=
0.
\end{equation} 
\end{proposition}

\begin{proof}
If $\nu=0$ choose $\tau_*=0$. If $\norm\nu>0$ 
consider the sets
\begin{equation}
\Delta
=
\{\widebar\sigma\in\Sigma(\A,\nu):\widebar\nu(\widebar\sigma)
	\ge
(1-t/2)\norm\nu\}
\qqtext{and}
\F
	=
\Big\{\widebar\sigma\in\Sigma(\A,\nu):
\sup_{\mu\in\Meas}\widebar\mu(\widebar\sigma)
	>
0\Big\}.
\end{equation}
Since $\F$ satisfies properties \tiref a and \tiref b
of Lemma \ref{lemma Boole1}, the order $>_\F$
may be defined. We claim that $\Delta$ admits 
$\widebar\tau$ such that 
$\widebar\tau\not
	>_\F
\widebar\tau'$ 
for any $\widebar\tau'\in\Delta$. Let to this 
end $\Delta_0$ be a maximal, linearly $>_\F$ 
ordered subset of $\Delta$. If $\Delta_0$ 
admits a $>_\F$ minimum, the claim is proved. 
If not, then by Lemma \ref{lemma Boole2} we 
may assume that $\Delta_0$ admits a 
countable subset having the same bounds 
as $\Delta_0$. Given that $\Delta_0$ is 
linearly $>_\F$ ordered we can extract an 
$>_\F$ decreasing sequence $\seqn{\widebar\sigma}$ 
from $\Delta_0$ such that $\Delta_0$ has the 
same lower bounds as 
$\{\widebar\sigma_n:n\in\N\}$. 
Upon passing to a subsequence, if necessary, 
we can further assume 
\begin{equation}
\label{cauchy}
\sup_{k>n}\widebar\nu(\widebar\sigma_n\sim\widebar\sigma_k)
<
2^{-2n}.
\end{equation}
For each $n\in\N$, choose $\sigma_n\in\widebar\sigma_n$ 
so that $\sigma_n\ge\sigma_{n+1}$. Given that
$\sigma_n\in\Sigma(\A,\nu)$ the quantity $\lim_k\nu(\sigma_n(k))$
exists for each $n\in\N$. But then, exploiting a 
diagonal argument, we can construct a sequence
$\seq ik$ of integers such that $i_k>i_{k-1}\vee k$
and that
\begin{equation}
\label{unif cauchy}
\sup_{\{n,p,q:\ n\le k\le p\wedge q\}}
\nu(\sigma_n(i_p)\bigtriangleup\sigma_n(i_q))
	\le
2^{-2k}
\qquad
k\in\N.
\end{equation}
Letting $\sigma'_n(k)=\sigma_n(i_k)$ we conclude
that $\sigma'_n\ge\sigma'_{n+1}$, 
$\sigma'_n\in\Sigma(\A,\nu)$ and that
$\widebar m(\widebar\sigma'_n)=\widebar m(\widebar\sigma_n)$
for all $n\in\N$ and all additive set function $m$ on $\A$.

Define now $\tau\in\Seq\A$ by letting
$
\tau(n)
	=
\sigma'_n(n)
$ 
for all $n\in\N$. Then $\tau(k)\subset\sigma'_n(k)$ for 
each $k\ge n$ so that 
$\widebar\tau
\le
\widebar\sigma'_n$.
To show that $\tau$ is the desired lower bound we need
to show that $\tau\in\Delta$. If $k>n$
\begin{align*}
\nu\big(\tau(n)\bigtriangleup \tau(k)\big)
	&=
\nu\big(\sigma'_n(n)\bigtriangleup \sigma'_k(k)\big)
\\
	&\le
\sup_{j>k}\nu\big(\sigma'_n(n)\bigtriangleup\sigma'_n(j)\big)
+
\sup_{j>k}\nu\big(\sigma'_k(k)\bigtriangleup\sigma_k(j)\big)
+
\lim_j\nu\big(\sigma'_n(j)\setminus\sigma'_k(j)\big)
\\
	&\le
2^{-2(n-1)}
\end{align*}
by \eqref{cauchy} and \eqref{unif cauchy} so that
$\tau\in\Seq{\A,\nu}$.
In addition, the inequality
\begin{align*}
\widebar\nu(\widebar\tau)
	\ge
\lim_n\lim_j\nu(\sigma'_n(j))
-
\lim_n\sup_{j>n}\nu(\sigma'_n(n)\bigtriangleup \sigma'_n(j))
	=
\lim_n\widebar\nu(\widebar\sigma'_n)
	\ge
(1-t/2)\norm\nu
\end{align*}
which follows from \eqref{unif cauchy} implies that 
$\widebar\tau\in\Delta$. Thus $\tau$ is a lower 
bound for $\Delta_0$ and, since $\Delta_0$ is 
maximal, it admits no $\widebar\upsilon\in\Delta$ 
with $\widebar\tau>_\F\widebar\upsilon$. This
conclusion translates into the statement
\begin{equation}
\label{claim}
\sup_{\mu\in\Meas}\widebar\mu(\widebar\tau\sim\widebar\upsilon)
	=
0
\qquad
\widebar\upsilon\in\Delta,\ 
\widebar\upsilon\le\widebar\tau.
\end{equation}
Choose, e.g., $\tau_0\in\SeqDn\A$ such that 
$\nu(\tau_0(n))<2^{-2n}$ and let
$\upsilon=\tau\sim\tau_0$.
Then $\upsilon\in\Seq{\A,\nu}$, $\upsilon\le\tau$ 
and
$
\widebar\nu(\upsilon)
	=
\widebar\nu(\tau\sim\tau_0)
	=
\widebar\nu(\tau)
	\ge 
(1-t/2)\norm\nu
$
i.e. $\widebar\upsilon\in\Delta$. But then by \eqref{claim}
\begin{align*}
0
	=
\sup_{\mu\in\Meas}\widebar\mu(\widebar\tau\sim\widebar\upsilon)
	=
\sup_{\mu\in\Meas}
\widebar\mu(\widebar\tau\wedge\widebar\tau_0)
	=
\sup_{\mu\in\Meas}\lim_j\lim_k
\mu(\tau(k)\cap\tau_0(j)).
\end{align*}
The same conclusion holds {\it a fortiori} if 
we replace $\widebar\tau$ with 
$\widebar\tau_*\in\Sigma^{\downarrow}(\A)$ 
chosen, in accordance with Lemma \ref{lemma Sn}, 
such that $\widebar\tau_*\le\widebar\tau$
and
$\widebar\nu(\widebar\tau_*)
	\ge
(1-t)\norm\nu$.
This leads to 
\begin{align*}
0
	=
\sup_{\mu\in\Meas}\lim_j\lim_k
\mu(\tau_*(k)\cap \tau_0(j))
	=
\sup_{\mu\in\Meas}\lim_j\lambda_\mu(\tau_0(j))
\end{align*}
where we have implicitly defined $\lambda_\mu\in ba(\A)$
via
\begin{equation}
\lambda_\mu(H)
	=
\lim_k\mu(\tau_*(k)\cap H)
\qquad
H\in\A.
\end{equation} 
According to Orlicz \cite[Theorem 3, p. 124]{orlicz} 
this is enough to conclude that $\lambda_\mu\ll\nu$.
However, by construction, $\lambda_\mu\ll\mu$. 
Then necessarily,
$
0
	=
\sup_{\mu\in\Meas}\lambda_\mu(1)
	=
\sup_{\mu\in\Meas}\widebar\mu(\widebar\tau_*)
$
for all $\mu\in\Meas$.
\end{proof}

It will be clear after the next result that the condition 
stated in Proposition \ref{propo D} is not only necessary 
for property \D\ but sufficient as well.

\begin{theorem}
\label{th D}
$\mathscr M$ possesses property \D\ 
if and only if it is dominated. 
\end{theorem}

\begin{proof}
Necessity has already been proved. To prove 
sufficiency, consider the collection $D$ of all 
pairs $(M,M')$ of subsets of $\Meas$ with
$M\subset M'$. For given $(M,M'),(N,N')\in D$, 
write $(M,M')\le(N,N')$ whenever $M'\subset N$. 
Since this defines a partial order, consider the 
maximal linearly ordered family 
$\{(M_\alpha,M'_\alpha):\alpha\in\Aa\}
	\subset 
D$ 
such that for each $\alpha\in\Aa$
\tiref{a}
$M_\alpha,M'_\alpha$ are countable and 
\tiref{b}
 there exists 
$0<\nu_\alpha\perp M_\alpha$ and 
$\mu'_\alpha\in M'_\alpha$ such that 
$\nu_\alpha\le\mu'_\alpha$.
If $\Meas$ possesses property \D, then according 
to Proposition \ref{propo D} for each $\alpha\in\Aa$ 
there exists 
$\tau_\alpha
	\in
\SeqDn\A$ 
such that
\begin{align}
\label{s pro}
\widebar\nu_\alpha(\widebar\tau_\alpha^c)
	<
\varepsilon\norm{\nu_\alpha}
\qtext{while}
\sup_{\{\mu\in\Meas:\ \mu\perp\nu_\alpha\}}
\widebar\mu(\widebar\tau_\alpha)
	=
0.
\end{align}
By construction, each $\alpha\in\Aa$ admits
countably  many predecessors, 
$\alpha_1,\alpha_2,\ldots$ and for each of these 
it is possible to construct $\tau_{\alpha_j}\in\SeqDn\A$
as in \eqref{s pro}. We define then 
$\sigma_\alpha\in\SeqDn\A$ by letting
\begin{equation}
\sigma_\alpha(n)
	=
\tau_\alpha(n)
\setminus
\bigcup_{j<n}\tau_{\alpha_j}(k^\alpha_j\vee n)
\qquad
n\in\N
\end{equation}
where, exploiting $\nu_\alpha\perp\nu_{\alpha_j}$, 
$k^\alpha_j$ is chosen so that
\begin{align*}
\sup_{k\ge k^\alpha_j}\nu_\alpha\big(\tau_{\alpha_j}(k)\big)
	<
2^{-j-1}\big[\varepsilon\norm{\nu_\alpha}
-
\widebar\nu_\alpha(\widebar\tau_\alpha^c)\big]
\qquad
j\in\N.
\end{align*}
Notice that
$\sigma_\alpha\le\tau_\alpha$ 
and that 
\begin{equation*}
\sigma_\alpha(n)\cap \tau_{\alpha_j}(n)
	\subset
\tau_{\alpha_j}(n)
\setminus
\tau_{\alpha_j}(k_j\vee n)
	=
0
\qquad
n\ge k^\alpha_j
\end{equation*}
so that
$\sigma_\alpha$ and $\tau_{\alpha_j}$
are quasi disjoint and, {\it a fortiori}, so are 
$\sigma_\alpha$ and $\sigma_{\alpha_j}$.
Moreover,
\begin{align*}
\nu_\alpha\big(\sigma_\alpha(n)^c\big)
	\le
\nu_\alpha\big(\tau_\alpha(n)^c\big)
+
\sum_{j\le n}\nu_\alpha\big(\tau_{\alpha_j}(k^\alpha_j\vee n)\big)
	\le
\nu_\alpha\big(\tau_\alpha(n)^c\big)
+
\frac12\big[\varepsilon\norm{\nu_\alpha}
-
\widebar\nu(\widebar\tau_\alpha^c)\big]
	<
\varepsilon\norm{\nu_\alpha}.
\end{align*}
The collection
$\{\sigma_\alpha:\alpha\in\mathfrak A\}
	\subset
\SeqDn\A$
is thus pairwise quasi disjoint and satisfies 
\eqref{s pro}. By property \D, $\Aa$ must be
countable. Let 
$M
	=
\bigcup_{\alpha\in\Aa}M'_\alpha\in D$.
If one could find $\mu\in\Meas$ and 
$0<\nu\le\mu$ such that $\nu\perp M$,
then the pair $(M,M\cup\{\mu\})$ would
contradict the maximality of 
$\{(M_\alpha,M'_\alpha):\alpha\in\Aa\}$.
Thus each $\mu\in\Meas$ is dominated 
by some $m\in M$ and, {\it a fortiori}, by 
the $\sigma$-convex combination of its 
elements.
\end{proof}

If $\Meas$ is dominated it is then clear
by Lemma \ref{lemma AL} that a dominating
measure is of the form 
\begin{equation}
\label{mu0}
\mu_0
	=
\sum_na_n\mu_n
\qqtext{for some}
\mu_1,\mu_2,\ldots\in\Meas,\ 
a_1,a_2,\ldots\in\R_+.
\end{equation}
Therefore $\nu\perp\Meas$ if and only if 
$\nu\perp\mu_0$. But then for every 
$0<t\le1$ there exists $\sigma\in\SeqDn\A$ 
such that
\begin{equation}
\widebar\nu(\widebar\sigma)
	\ge
(1-t)\norm\nu
\qtext{while}
\widebar\mu_0(\widebar\sigma)
	=
\sup_{\mu\in\Meas}\widebar\mu(\widebar\sigma)
	=
0.
\end{equation}
In other words after Theorem \ref{th D} 
the condition of Proposition \ref{propo D} 
is sufficient for property \D.

\section{Weak compactness in the space 
of additive set functions.}
\label{sec D*}

Let us now consider the case in which $\A$ is
an algebra of subsets of some non empty 
set $\Omega$ and $\Meas\subset ba(\A)$. 
In the special case in which $\Meas$ is norm 
bounded, uniform strong additivity is 
equivalent to relative weak compactness (see 
\cite{brooks}) and implies that $\Meas$ must 
be dominated. This implication is true even 
without norm boundedness.

\begin{corollary}
A uniformly strongly additive set $\Meas$ is  
dominated.
\end{corollary}

\begin{proof}
Suppose that $\Meas$ fails to possess property
\D. Then it is possible to find $\eta>0$ and a 
pairwise quasi disjoint sequence $\seq\sigma k$ 
in $\SeqDn\A$ such that
$\inf_k\sup_{\mu\in\Meas}
\widebar{\abs\mu}(\widebar\sigma_k)
	>
\eta$. By picking $B_k$ from the sequence $\sigma_k$ 
for each $k\in\N$ accurately we can then form a pairwise 
disjoint sequence $\seq Bk$ such that 
$
\inf_k\sup_{\mu\in\Meas}
\abs\mu(B_k)
	>
\eta
$
so that uniform strong additivity fails.
\end{proof}

The connection between property \D\ and
weak compactness is made precise in the 
following:

\begin{theorem}
\label{th D*}
$\Meas$ is relatively weakly compact if and only
norm bounded and of class \Du.
\end{theorem}

\begin{proof}
The set $\{\nu\}$ is trivially of class \Du. If
$\nu$ dominates $\Meas$ uniformly, then 
$\Meas$ is of class \Du. Thus relative weak 
compactness implies property \Du.

To prove the converse, denote by $\sigma\A$
the $\sigma$ algebra generated by $\A$ and,
if $m\in ba(\A)_+$, by $m_*$ the set function
on $\sigma\A$ defined by 
\begin{equation}
m_*(B)
=
\sup_{\{A\in\A:A\subset B\}}m(A)
\qquad
B\in\sigma\A.
\end{equation}
We first show that the collection 
$\{\abs\mu_*:\mu\in\Meas\}$
itself possesses property \Du. In fact, if 
$\{\sigma_\alpha:\alpha\in\Aa\}$
is a pairwise quasi disjoint family in 
$\SeqDn{\sigma\A}$ such that
\begin{equation*}
\lim_n\sup_{\mu\in\Meas}\abs
\mu_*(\sigma_\alpha(n))
	>
0
\qquad
\alpha\in\Aa
\end{equation*}
for each $\varepsilon>0$, $\alpha\in\Aa$ and 
$n\in\N$ we can find $\tau_\alpha(n)\in\A$ such 
that $\tau_\alpha(n)\subset\sigma_\alpha(n)$
and 
$\sup_{\mu\in\Meas}
\abs\mu_*\big(\sigma_\alpha(n)\setminus\tau_\alpha(n)\big)
\le
\varepsilon2^{-n}$. Let
$\tau'_\alpha(n)=\bigcap_{j=1}^n\tau_\alpha(j)$.
Then $\tau'_\alpha\in\SeqDn\A$, 
$\tau'_\alpha\le\sigma_\alpha$
and
\begin{align*}
\sup_{\mu\in\Meas}
\abs\mu\big(\tau'_\alpha(n)\big)
	&\ge
\sup_{\mu\in\Meas}
\abs\mu\Big(\bigcap_{j=2}^n\tau_\alpha(j)\Big)
-
\sup_{\mu\in\Meas}
\abs\mu\Big(\bigcap_{j=2}^n\tau_\alpha(j)
\setminus\tau_\alpha(1)\Big)
\\
	&\ge
\sup_{\mu\in\Meas}
\abs\mu\Big(\bigcap_{j=2}^n\tau_\alpha(j)\Big)
-
\sup_{\mu\in\Meas}
\abs\mu_*\big(\sigma_\alpha(1)\setminus\tau_\alpha(1)\big)
\\
	&\ge
\sup_{\mu\in\Meas}
\abs\mu\Big(\bigcap_{j=2}^n\tau_\alpha(j)\Big)
-\varepsilon 2^{-1}
\\
	&\ge
\sup_{\mu\in\Meas}
\abs\mu_*\big(\sigma_\alpha(n)\big)
-\varepsilon \sum_{j=1}^n2^{-j}.
\end{align*}
This shows that $\{\tau'_\alpha:\alpha\in\Aa\}$
forms a pairwise quasi disjoint family of decreasing 
sequences that satisfies the condition
$
\lim_n\sup_{\mu\in\Meas}\abs
\mu\big(\tau'_\alpha(n)\big)
	>
0
$
for each $\alpha\in\Aa$. By property \Du, 
$\Aa$ must then be countable, thus proving
the preceding claim.

Take a disjoint sequence $\upsilon\in\Seq\A$
and define
\begin{equation}
\upsilon(E)=\bigcup_{k\in E}\upsilon(k)
\qtext{and}
\psi(E)
	=
\sup_{\mu\in\Meas}\abs
\mu_*\big(\upsilon(E)\big)
\qquad
E\subset\N.
\end{equation}
Given that $\Meas$ is norm bounded, that the 
sequence is disjoint and that
$\upsilon(E)\in\sigma\A$ 
we conclude that $\psi:\mathcal P(\N)\to\R_+$
is monotone, $\psi(\emp)=0$ and that $\psi$ is 
of class \D. As shown in Example \ref{ex N}, 
there exists an infinite set $E_0\subset\N$ 
such that, letting $E_n=E_0\cap\{n,n+1,\ldots\}$,
\begin{equation}
\lim_n\psi(E_n)
	=
\lim_n\sup_{\mu\in\Meas}
\abs\mu_*\big(\upsilon(E_n)\big)
	=
0.
\end{equation}
Upon passing to a subsequence if necessary,
we can assume the existence of 
$i_n\in E_n\setminus E_{n+1}$.
Then,
\begin{align*}
\lim_n\sup_{\mu\in\Meas}\abs\mu(\upsilon(i_n))
	\le
\lim_n\psi(E_n)
	=
0.
\end{align*}
This rules out the possibility that
$\limsup_k\sup_{\mu\in\Meas}\abs\mu(\upsilon(n))
>
0$
and proves that $\Meas$ is uniformly strongly
additive and thus relatively weakly compact.
\end{proof}

We deduce easily the following special version 
of a result of Zhang \cite[Theorem 1.3]{zhang}.
In this claim it is essential to take $\Meas$ to
consist of positive set functions.

\begin{corollary}
\label{cor D*}
A weakly$^*$ compact set $\Meas\subset ba(\A)_+$ 
is weakly compact if and only if of class \D.
\end{corollary}

\begin{proof}
If $\Meas$ is weakly$^*$ compact it is then
weakly closed and bounded. By Theorem
\ref{th D*} it remains to prove that is of class 
\Du. But for a weakly$^*$ compact set
of positive, additive set functions this is 
equivalent to property \D, by virtue of Dini's 
Theorem.
\end{proof}

\section{Relation with the literature}
\label{sec close}

When $\Meas$ is the set of all measures
on $\A$, properties \D\ and \CC\ are rightfully 
interpreted as properties of the algebra $\A$. 
Given that each element of $\A$ other than $0$
is assigned positive mass by some measure, 
property \D\ is sufficient to imply the existence 
of a set function that vanishes only on $0$, i.e. 
that $\A$ is a measure algebra. Maharam 
conjectured that the \CC\ condition may 
possibly be sufficient for a Boolean algebra 
to be measure algebra (see also 
\cite[Theorem 2.4]{horn_tarski}). Gaifman 
\cite{gaifman} later constructed an example of 
a Boolean algebra satisfying the \CC\ condition 
but failing to be a measure algebra. Quite 
recently, Talagrand \cite{talagrand} provided 
an example of a Boolean $\sigma$ algebra 
satisfying the \CC\ property and the so-called 
weak distributive law but which is not a measure 
algebra. A necessary and sufficient condition 
has been given by Kelley \cite{kelley}.

We can show a special case of a fairly general
Boolean algebra in which the \CC\ property
is necessary and sufficient to be a measure 
algebra.

\begin{theorem}
\label{th Gaifman}
Let $\A$ be a Boolean algebra. $\Sigma(\A)$ 
is a measure algebra if and only if it possesses 
property \CC.
\end{theorem}

\begin{proof}
Measure algebras possess the \CC\ property. 
Conversely, if $\{\sigma_\alpha:\alpha\in\Aa\}$
is a pairwise quasi disjoint family then 
$\{\widebar\sigma:\alpha\in\Aa\}$ is disjoint 
in $\Sigma(\A)$; moreover
$\sup\widebar m(\widebar\sigma_\alpha)
	>
0$,
the supremum being over all measures on $\A$,
is equivalent to $\widebar\sigma_\alpha\ne 0$ in 
$\Sigma(\A)$. If $\Sigma(\A)$ satisfies the \CC\ 
property, $\Aa$ must be countable so that $\A$
has property \D\ and the family of all measures 
on $\A$ is dominated by some $\mu_0$, by 
Theorem \ref{th D}. Then, 
$\widebar\mu_0(\widebar\sigma)=0$ if and 
only if $\widebar m(\widebar\sigma)=0$ for
all measures $m$ on $\A$, i.e. if $\widebar\sigma=0$.
\end{proof}

\end{document}